\documentclass[10pt,a4paper,reqno]{amsart}
\usepackage[T1]{fontenc}
\usepackage[utf8]{inputenc}
\usepackage[american]{babel}

\usepackage{amsmath,amssymb,amsthm}
\usepackage{mathtools}
\usepackage[scr=boondox]{mathalfa}
\usepackage{cases}
\usepackage{amsfonts}
\usepackage{mathrsfs}

\usepackage{graphicx}
\usepackage[font=small]{caption}
\usepackage{subcaption}
\usepackage{multirow}
\usepackage[shortlabels]{enumitem}
\usepackage{cases}

\usepackage{xcolor}
\usepackage[colorlinks=true,linkcolor=blue]{hyperref}
\makeatletter 
\@namedef{subjclassname@2020}{\textup{2020} Mathematics Subject Classification}
\makeatother

\theoremstyle{plain}
\newtheorem{thm}{Theorem}
\newtheorem{prop}[thm]{Proposition}
\newtheorem{lem}[thm]{Lemma}

\theoremstyle{definition}
\newtheorem*{thm*}{Theorem}

\theoremstyle{remark}
\newtheorem{rem}[thm]{Remark}
\newtheorem*{rem*}{Remark}

\theoremstyle{plain}

\numberwithin{equation}{section}

\newtheorem*{question*}{\itshape QUESTION}

\theoremstyle{plain}

\newcommand{\mb}{\mathbb}
\newcommand{\mc}{\mathcal}

\newcommand{\mf}{\mathfrak}

\newcommand{\N}{\mb{N}}
\newcommand{\R}{\mb{R}}
\newcommand{\C}{\mb{C}}

\newcommand{\g}{{\mf{g}}}

\newcommand{\lsum}[1]{\sum_{n=#1}^{\infty}}

\newcommand{\abs}[1]{\left|#1\right|}
\newcommand{\bigO}[1]{\mc{O}\left(#1\right)}
\newcommand{\paren}[1]{\left( #1 \right)}

\newcommand{\dd}{\,\mathrm{d}}
\newcommand{\ddiv}{\mathrm{div}}
\newcommand{\shift}{\mathrm{S}}

\newcommand{\grad}{\nabla}
\newcommand{\lap}{\Delta}

\begin{document}

\title[]{The $p$-Hardy--Rellich--Birman inequalities on the half-line}

\author{Franti\v sek \v Stampach}
\address[Franti{\v s}ek {\v S}tampach]{
  Department of Mathematics, Faculty of Nuclear Sciences and Physical Engineering, Czech Technical University in Prague, Trojanova~13, 12000 Praha~2, Czech Republic
}
\email{stampfra@cvut.cz}

\author{Jakub Waclawek}
\address[Jakub Waclawek]{
  Department of Mathematics, Faculty of Nuclear Sciences and Physical Engineering, Czech Technical University in Prague, Trojanova~13, 12000 Praha~2, Czech Republic
}
\email{waclajak@cvut.cz}

\subjclass[2020]{26D15, 39A12, 47J05}

\keywords{Hardy inequality, Rellich inequality, Birman inequality, Copson inequality}

\date{\today}

\begin{abstract}
The classical discrete $p$-Hardy inequality establishes a sharp relationship between the $\ell^{p}$-norms of a sequence and its discrete derivative. In this paper, we generalize this inequality to discrete derivatives of arbitrary integer order $\ell \geq 1$, yielding discrete $p$-Rellich ($\ell=2$) and general $p$-Birman ($\ell \geq 3$) inequalities. As a key step in the proof, we deduce a variant of the Copson inequality with a negative exponent, which may be of independent interest. Furthermore, we demonstrate how the continuous $p$-Birman inequality can be recovered from our discrete version, providing an alternative proof of this classical result. All constants in the obtained inequalities are shown to be optimal.
\end{abstract}

\maketitle

\section{Introduction and main results}

\subsection{Introduction}
More than a hundred years ago, Hardy discovered his celebrated inequality, which, in the difference form, reads
\begin{equation}
  \label{eq:p-Hardy_dis}
  \sum_{n=1}^\infty \abs{u_n-u_{n-1}}^p
    \ge \paren{\frac{p-1}{p}}^{\!p} \lsum{1} \frac{\abs{u_n}^p}{n^p}
\end{equation}
and holds for all compactly supported sequences $u\in C_0(\N_0)$ satisfying $u_0=0$ and $p>1$. In fact, Hardy had already obtained an analogue of \eqref{eq:p-Hardy_dis} with a smaller constant as early as 1915. In the years that followed, several other prominent mathematicians made substantial contributions to this topic. In particular, Landau proved \eqref{eq:p-Hardy_dis} with its sharp constant. We refer to \cite{kufner-malingrada-oersson_06-prehistory, kufner-maligranda-persson_07-hardy_ineq} for a historical overview and proofs.  The continuous analogue of \eqref{eq:p-Hardy_dis}, discovered shortly afterward, reads
\begin{equation}
 \label{eq:p-Hardy_cont}
  \int_0^\infty \abs{\varphi'(x)}^p\dd x
    \ge \paren{\frac{p-1}{p}}^{\!p} \int_0^\infty \frac{\abs{\varphi(x)}^p}{x^p}\dd x
\end{equation}
for smooth functions $\varphi\in C^\infty_0(\R_{+})$ compactly supported in $\R_{+}\equiv(0,\infty)$.

Given the wide range of applications of the Hardy inequality in probability theory, the theory of partial differential equations, and numerous other domains, its various generalizations have been the subject of extensive investigation. Most of the research is focused on the particular case $p=2$ when inequalities \eqref{eq:p-Hardy_dis} and \eqref{eq:p-Hardy_cont} represent lower bounds on the discrete and continuous Dirichlet Laplacian on the half-line in the sense of quadratic forms. In the context of the present work, we focus on higher-order variants of inequalities \eqref{eq:p-Hardy_dis} and \eqref{eq:p-Hardy_cont}, which we refer to as the \textit{Birman inequalities}. In 1961, Birman~\cite{birman_61} generalized inequality~\eqref{eq:p-Hardy_cont}, with $p=2$, to higher order derivatives: 
\begin{equation}
  \label{eq:Birman_con}
  \int_0^\infty \abs{\varphi^{(\ell)}(x)}^2\dd x
    \ge \paren{\frac{1}{2}}_{\!\ell}^{\!2} \int_0^\infty \frac{\abs{\varphi(x)}^2}{x^{2\ell}}\dd x,
\end{equation}
where $\varphi\in C^\infty_0(\R_{+})$, $\ell\in\N$, and $(a)_\ell \coloneq a(a+1)\dots(a+\ell-1)$ is the Pochhammer symbol. A proof of~\eqref{eq:Birman_con} can be found in \cite[pp.~83--84]{glazman}; see also the recent paper \cite{ges-lit-mic-wel_18} for other proofs and generalizations. The particular case $\ell=2$ of \eqref{eq:Birman_con} was discovered earlier by Rellich~\cite{rellich_56}, even in higher dimensions, and is commonly referred to as the \textit{Rellich inequality}.

Surprisingly, the discrete version of \eqref{eq:Birman_con} was proven only recently in~\cite{huang-ye_24-hardy}; the discrete Rellich inequality appears also in~\cite{ger-kre-sta_25}. In the discrete setting, the higher-order derivative on the left-hand side of~\eqref{eq:Birman_con} is to be replaced by an integer power of the discrete gradient; see Section~\ref{subsec:notation} for the notation. Then, for any $\ell\in\N$, the discrete Birman inequality can be expressed as
\begin{equation}
  \label{eq:discrete_Birman_ineq}
  \lsum{\ell} \abs{\grad^\ell u_n}^2
    \ge \paren{\frac{1}{2}}_{\!\ell}^{\!2} \lsum{\ell} \frac{\abs{u_n}^2}{n^{2\ell}},
\end{equation}
where $u\in C_{0}(\N_{0})$ with $u_n=0$ for all $n<\ell$.

For $p\neq2$, the literature is less complete. The continuous $p$-Birman inequality
\begin{equation}
\label{eq:p-Birman_con}
 \int_0^\infty \abs{\varphi^{(\ell)}(x)}^{p}\dd x \ge \paren{\frac{p-1}{p}}_{\!\ell}^{\! p} \int_0^\infty \frac{\abs{\varphi(x)}^p}{x^{\ell p}}\dd x
\end{equation}
holds for $p>1$, $\ell\in\N$, and $\varphi\in C^\infty_0(\R_{+})$. 
Although we could not locate an explicit reference, with the exception of the $p$-Rellich inequality ($\ell=2$) which appears in \cite[Eq.~(1.2)]{owen_99-hardy-rellich}, inequality \eqref{eq:p-Birman_con} is well known to experts and is fundamentally rooted already in the original work of Hardy; see Remark \ref{rem:fritz}. In contrast, its discrete counterpart appears to remain unknown for general $p>1$. The primary goal of this paper is to establish the discrete $p$-Birman inequality with its optimal constant, thereby completing the picture of $p$-Birman inequalities on the half-line for general $p$>1.

\subsection{Organization of the paper}

The rest of the paper is structured as follows. After introducing the necessary notation, we present our main results: the discrete $p$-Birman and weighted $p$-Hardy inequalities. These inequalities are proven in Section~\ref{sec:proofs}. In Section~\ref{sec:con_birman}, we utilize the discrete results to recover their continuous counterparts, which is used in Section~\ref{sec:optimality} to prove sharpness of all constants. Finally, we briefly discuss current research on stronger optimality of the inequalities under consideration.

\subsection{Notation for discrete derivatives}\label{subsec:notation}


We define \textit{discrete gradient} and \textit{divergence} as difference operators acting on the space of sequences $C(\N_0)$ indexed by $\N_{0}$ as
\begin{equation}
  \label{eq:def_div_grad}
  \grad u_n \coloneq u_n-u_{n-1}
  \quad\text{ and }\quad
  \ddiv u_n \coloneq u_{n+1}-u_n,
\end{equation}
with the convention that $u_{-1}\coloneq0$, i.e. $\grad u_0=u_0$. Analogously, when composing difference operators, we assume vanishing values of sequences if the index is negative, i.e., we first formally apply the composition for all indices $n\in\N_0$ and then set $u_n\coloneq0$ whenever $n<0$. With this convention, the discrete gradient and divergence commute and their composition defines the \textit{discrete Laplace operator} $\lap\coloneq\ddiv\circ\grad$ on $C(\N_0)$ by the formula
\begin{equation*}
  \lap u_n = u_{n-1}-2u_n+u_{n+1},
\end{equation*}
for all $n\in\N_0$. By the convention, $\lap u_0=u_1-2u_0$.

\begin{rem*}
Previous works~\cite{huang-ye_24-hardy,stampach-waclawek_24-birman} formulate the discrete Birman inequality~\eqref{eq:discrete_Birman_ineq} with the gradient term $\grad^{\ell}u_n$ replaced by the fractional Laplacian
$(-\lap)^{\ell/2}u_n$, summed over $n\geq\lceil \ell/2 \rceil$, where $(-\lap)^{\ell/2}$ is the composition of $\ell/2$ discrete Laplacians $-\Delta$ if $\ell$ is even, while $(-\lap)^{\ell/2} \coloneq \grad\circ(-\lap)^{(\ell-1)/2}$ if $\ell$ is odd.
With our convention, $(-\lap)^{\ell/2}u_n=(-1)^{\lfloor\ell/2\rfloor} \grad^\ell u_{n+\lfloor\ell/2\rfloor}$ for all $n\in\N_0$, making both formulations equivalent. To avoid unnecessary use of fractional powers and floor/ceiling notation in this paper, we adhere to the simpler representation as in \eqref{eq:discrete_Birman_ineq}.
\end{rem*}

\subsection{Main results}

Our main result is the discrete $p$-Birman inequality on the integer half-line.

\begin{thm}[discrete $p$-Birman inequality]
  \label{thm:p-birman_dis}
  Let $\ell\in\N$ and $p>1$. Then for all $u\in C_{0}(\N_{0})$ such that $u_{n}=0$ for $n<\ell$, the inequality
  \begin{equation}
    \label{eq:p-birman_dis}
    \lsum{\ell} \abs{\grad^\ell u_n}^p
      \ge B_{p}^{(\ell)} \lsum{\ell} \frac{\abs{u_n}^p}{n^{\ell p}}
  \end{equation}
  holds with the optimal constant 
  \begin{equation}
  \label{eq:p-birman_constant}
  B_{p}^{(\ell)}\coloneq\paren{\frac{p-1}{p}}_{\!\ell}^{\!p}.
  \end{equation}
\end{thm}

\begin{rem}
 Inequality~\eqref{eq:p-birman_dis} can be equivalently written in discrete integral form. By solving the difference 
 equation $\grad^{\ell}u_n = v_n$ for $n\geq\ell$, with the initial condition $u_n=0$ for $n<\ell$, one finds
 \[
 u_n=\sum_{\ell\leq k_1\leq\dots\leq k_{\ell}\leq n}v_{k_1}=\sum_{j=\ell}^{n}\scalebox{1.1}{$\binom{n+\ell-1-j}{\ell-1}$}\,v_j
 \]
 for $n\geq\ell$. Substituting into~\eqref{eq:p-birman_dis} yields the inequality
 \begin{equation}
 \label{eq:p-birman_dis_integral_form}
 \lsum{\ell} \abs{v_n}^p
      \ge B_{p}^{(\ell)} \lsum{\ell}\bigg|\frac{1}
      {n^{\ell}}\sum_{k=\ell}^{n}\scalebox{1.1}{$\binom{n+\ell-1-k}{\ell-1}$}\,v_k\bigg|^{p}
 \end{equation}
 for any $v\in C_0(\N)$. In fact, using Fatou's lemma, inequality \eqref{eq:p-birman_dis_integral_form} extends to all complex sequences $v$. After shifting the index, Theorem~\ref{thm:p-birman_dis} states that the generalized discrete Hardy operator
 \[
 \mathrm{H}^{(\ell)}v_{n}:=\frac{1}{(n+\ell-1)^{\ell}}\sum_{k=1}^{n}\scalebox{1.1}{$\binom{n+\ell-1-k}{\ell-1}$}\,v_k
 \]
 is bounded on $\ell^{p}(\N)$, for all $1<p<\infty$, with the operator norm equal to $1/(1-1/p)_{\ell}$.
\end{rem}

Inequality \eqref{eq:p-birman_dis} is proved by an iterative application of a weighted discrete $p$-Hardy inequality, similar to the classical \textit{Copson inequality}
\begin{equation} \label{eq:copson}
    \lsum{1} n^\alpha \abs{\grad u_n}^p \ge \paren{\frac{p-\alpha-1}{p}}^{\!p} \lsum{1} n^{\alpha-p} \abs{u_n}^p,
\end{equation}
where $u\in C_0(\N_0)$ with $u_0=0$ and $0\leq\alpha<p-1$. The non-negativity of $\alpha$ in~\eqref{eq:copson} is essential, as the inequality fails for $\alpha<0$, which is precisely the regime required for our proof of \eqref{eq:p-birman_dis}. If $\alpha<0$, a slightly smaller weight is required on the right-hand side. We present this version of Copson inequality in the next theorem, as it may be of independent interest. Copson's original work \cite{copson_28} contains a proof of~\eqref{eq:copson}, the optimality of the constant is established for example in \cite[Corollary 3]{bennett_87}.

\begin{thm}\label{thm:weighted_hardy}
  For all $p>1$, $\alpha<0$, and $u\in C_{0}(\N_0)$ with $u_0=0$, the inequality
  \begin{equation}\label{eq:p-Hardy_weight}
      \lsum{1} n^\alpha \abs{\grad u_n}^p \ge C_{p}(\alpha) \lsum{1} (n+1)^{\alpha-p} \abs{u_n}^p,
  \end{equation}
  holds with the optimal constant
   \begin{equation}
    \label{eq:weight_constant}
       C_{p}(\alpha)\coloneq\paren{\frac{p-\alpha-1}{p}}^p.
   \end{equation}
\end{thm}

The proofs of inequalities~\eqref{eq:p-birman_dis} and \eqref{eq:p-Hardy_weight} are developed in Section~\ref{sec:proofs} and optimality of the constants~\eqref{eq:p-birman_constant} and~\eqref{eq:weight_constant} are established in Section~\ref{sec:optimality}. Theorems~\ref{thm:p-birman_dis} and \ref{thm:weighted_hardy} generalize \cite[Theorems~1.1 and~1.2]{huang-ye_24-hardy} to the case of general $p>1$.

\section{Proofs of the inequalities} \label{sec:proofs}

\subsection{An abstract \texorpdfstring{$p$}{p}-Hardy inequality}

Our starting point is an abstract weighted $p$-Hardy inequality, whose proof has already been indicated in~\cite[Sec.~7.1]{huang-ye_24-hardy} and relies on an auxiliary inequality from~\cite[Lemma~2.6]{frank-seiringer_08-ineq}. For the reader's convenience and completeness, we restate the inequalities together with their short proofs.

\begin{lem}
  Let $p>1$. For any $t\in[0,1]$ and $z\in\C$, we have
  \begin{equation}
    \label{eq:lem:ineq}
    \abs{z-t}^p \ge (1-t)^{p-1}(\abs{z}^p-t).
  \end{equation}
\end{lem}

\begin{proof}
  The inequality is trivially true for $t\in\{0,1\}$. Fix $t\in(0,1)$. It is easy to see that, for fixed $\abs{z}$, the minimum of the left-hand side of \eqref{eq:lem:ineq} is attained for $z>0$. Since the right-hand side is non-positive for $z\le t^{1/p}$, it remains to check the inequality for $z>t^{1/p}$. Dividing~\eqref{eq:lem:ineq} by $z^p-t>0$, we obtain
  \begin{equation*}
    \frac{(z-t)^p}{z^p-t} \ge (1-t)^{p-1}.
  \end{equation*}
 Computing the derivative with respect to $z$ of the left-hand side shows that this function has the global minimum in $(t^{1/p},\infty)$ attained at $z=1$ with the value $(1-t)^{p-1}$, which completes the proof.
\end{proof}

\begin{prop}\label{prop:weighted-hardy}
  Let $p>1$, $V_n\ge0$ for all $n\in\N$, $\g_{n+1}\geq\g_n>0$ for all $n\in\N$, and $\g_0=0$. Then for any $u\in\C_{0}(\N_0)$ with $u_0=0$, we have the inequality
  \begin{equation}
    \label{eq:prop:weighted-hardy}
    \lsum{1} V_n \abs{\grad u_n}^p
      \ge - \lsum{1} \frac{\ddiv(V(\grad\g)^{p-1})_n}{\g_n^{p-1}} \abs{u_n}^p.
  \end{equation}
\end{prop}

\begin{proof}
First, we show that, under the assumptions, the inequality 
\begin{equation}
  \label{eq:pr:prop-1}
  \abs{u_n-u_{n-1}}^p
    \ge \paren{\frac{\abs{u_n}^p}{\g_n^{p-1}}-\frac{\abs{u_{n-1}}^p}{\g_{n-1}^{p-1}}}
      \paren{\g_n-\g_{n-1}}^{p-1}
\end{equation}
holds for all $n\in\N$, where the term $\abs{u_0}^p/\g_0^{p-1}$ is to be interpreted as $0$. If $n=1$ or $u_{n-1}=0$, inequality~\eqref{eq:pr:prop-1} follows directly from the monotonicity of $\g$. For $n\ge2$ such that $u_{n-1}\neq0$, we temporarily set  $\varphi_n\coloneq u_n/\g_n$ and substitute for $z=\varphi_n/\varphi_{n-1}\in\C$ and $t=\g_{n-1}/\g_n\in(0,1]$ into~\eqref{eq:lem:ineq}. It yields the inequality
\begin{equation*}
  \abs{\frac{\varphi_n}{\varphi_{n-1}}-\frac{\g_{n-1}}{\g_n}}^p
    \ge \paren{1-\frac{\g_{n-1}}{\g_n}}^{p-1}
      \paren{\abs{\frac{\varphi_n}{\varphi_{n-1}}}^p-\frac{\g_{n-1}}{\g_n}}.
\end{equation*}
Multiplying this inequality by $\g_n^p\abs{\varphi_{n-1}}^p\ge0$ implies \eqref{eq:pr:prop-1}.

Now, by multiplying both sides of \eqref{eq:pr:prop-1} by $V_n\geq0$ and summing over $n\in\N$, we obtain
\[
  \lsum{1} V_n\abs{\grad u_n}^p
    \ge \lsum{1} V_n (\grad\g_n)^{p-1} \frac{\abs{u_n}^p}{\g_n^{p-1}}
      - \lsum{1} V_{n+1} (\grad\g_{n+1})^{p-1} \frac{\abs{u_n}^p}{\g_n^{p-1}},
\]
which is the inequality~\eqref{eq:prop:weighted-hardy}. The proof is complete.
\end{proof}

\subsection{Proof of the weighted \texorpdfstring{$p$}{p}-Hardy inequality \eqref{eq:p-Hardy_weight}}

We prove \eqref{eq:p-Hardy_weight} by choosing $V_n$ in Proposition~\ref{prop:weighted-hardy} as a power sequence and $\g_n$ as a suitable Gamma ratio. The specific choice of the parameter sequence $\g$ is essential to get a tight inequality.

\begin{proof}[Proof of inequality \eqref{eq:p-Hardy_weight}] Suppose $p>1$ and $\alpha<0$.
We will apply Proposition~\ref{prop:weighted-hardy} with the sequences
\begin{equation} \label{eq:def_V_g}
    V_{n}\coloneq 
    \begin{cases} 
    0 & \mbox{ if } n=1, \\
    (n-1)^{\alpha} & \mbox{ if } n\geq2, 
    \end{cases}
    \qquad 
    \g_{n}\coloneq 
    \begin{cases} 
    0 & \mbox{ if } n=0, \\
    \Gamma\bigl(n+1-\frac{\alpha+1}{p}\bigr)\big/\,\Gamma(n) & \mbox{ if } n\geq1, 
    \end{cases}
\end{equation}
where $\Gamma$ is the Euler Gamma function. Since for all $n\geq2$, we have
\[
\g_{n}-\g_{n-1}=\frac{p-\alpha-1}{p}\frac{\Gamma\bigl(n-\frac{\alpha+1}{p}\bigr)}{\Gamma(n)}>0,
\]
the sequence $\g$ is increasing. As other assumptions of Proposition~\ref{prop:weighted-hardy} are fulfilled trivially, we obtain the inequality   
\begin{equation*}
  \lsum{2} (n-1)^\alpha \abs{\grad u_n}^p \ge \lsum{2} \rho_n \abs{u_n}^p
\end{equation*}
for all $u\in C_{0}(\N_0)$ with $u_0=u_1=0$ (here we additionally require that $u_1=0$), where 
\begin{equation}\label{eq:def_rho}
  \rho_n\equiv -\frac{\ddiv(V(\grad\g)^{p-1})_n}{(\g_n)^{p-1}}
    = \paren{\frac{p-\alpha-1}{p}}^{p-1} n^{\alpha-p+1}\,H\!\left(\frac{1}{n}\right)
\end{equation}
and $H$ is an auxiliary function defined on interval $[0,1)$ by
\begin{equation*}
  H(x) \coloneq \frac{\paren{1-x}^\alpha}{\big(1-\frac{\alpha+1}{p}x\big)^{p-1}}-1.
\end{equation*}

In order to establish the inequality~\eqref{eq:p-Hardy_weight}, it suffices to show that $H$ satisfies
\begin{equation}
 \label{eq:ineq_H_aux}
  H(x) > \frac{p-\alpha-1}{p}\,x, \quad\forall x\in(0,1),
\end{equation}
which we verify in the rest of the proof. Since $H(0)=0$ and $H'(0)=(p-\alpha-1)/p$, inequality~\eqref{eq:ineq_H_aux} can be written as $H(x)>H(0)+H'(0)x$, and so \eqref{eq:ineq_H_aux} holds true if $H$ is strictly convex on $(0,1)$. A routine calculation yields
\[
  H''(x)= \frac{p-\alpha-1}{p^2}\frac{(1-x)^{\alpha-2}}{\big(1-\frac{\alpha+1}{p}x\big)^{p+1}}
    \left[(p-\alpha)(1+\alpha)^2x^2-2(1+\alpha)(p-\alpha)x+p(1-\alpha)\right].
\]
Since the discriminant $4\alpha(1+\alpha)^2(p-1)(p-\alpha)$ of the quadratic term is negative, we conclude that $H''(x)>0$ for all $x\in(0,1)$, indeed.
\end{proof}

\begin{rem}
The Copson inequality \eqref{eq:copson} can also be established using the approach described above. Assuming $0\leq\alpha< p-1$, we apply Proposition~\ref{prop:weighted-hardy} with $V_{n} \coloneq n^{\alpha}$ and $\g_n$ defined in \eqref{eq:def_V_g}. This yields the inequality
\begin{equation}
    \label{eq:improved_Copson}
    \lsum{1} n^\alpha |\grad u_n|^p \ge \lsum{1} \rho_n |u_n|^p,
\end{equation}
with $\rho_n$ given by \eqref{eq:def_rho}, but with the function $H$ redefined on $[0,1]$ by
\[
    H(x) \coloneq \left(1-\frac{\alpha+1}{p}x\right)^{1-p} - (1+x)^\alpha.
\]
Although this version of $H$ is no longer convex on $(0,1)$, inequality \eqref{eq:ineq_H_aux} remains true and can be verified by elementary means. The Copson inequality \eqref{eq:copson} then follows. 
Furthermore, inequality \eqref{eq:improved_Copson} offers an improved version of the Copson inequality, which can be of independent interest, as the existing literature on such refinements, for example
\cite{das-manna_23-weight_hardy, DasManna2025_ImprovedCopson, das-man-pau_25}, seems to consider the special case $p=2$ only.
\end{rem}

\subsection{Proof of the \texorpdfstring{$p$}{p}-Birman inequality \eqref{eq:p-birman_dis}}

Throughout this proof, we will make use of the forward shift operator $\shift$ acting on sequences by
\[
  \shift u_n \coloneq u_{n+1}, \quad n\in\N_{0}.
\]

\begin{proof}[Proof of inequality~\eqref{eq:p-birman_dis}]
The proof proceeds by induction in $\ell\in\N$. For $\ell=1$, inequality \eqref{eq:p-birman_dis} is the classical $p$-Hardy inequality \eqref{eq:p-Hardy_dis}.

Assume $\ell\ge2$ and pick $u\in C_0(\N_0)$ with $u_n=0$ for all $n<\ell$. Recalling definitions of the discrete gradient, divergence, and our convention, see Section \ref{subsec:notation}, we have equalities $\ddiv=\grad\circ\shift=\shift\circ\grad$ and $\ddiv\circ\grad=\grad\circ\ddiv$. It follows that
\begin{equation*}
    \lsum{\ell} \abs{\grad^\ell u_n}^p
        =\lsum{\ell-1}\abs{\grad^{\ell-1}\ddiv u_n}^p
        \ge B_p^{(\ell-1)} \lsum{\ell-1}\frac{\abs{\ddiv u_n}^p}{n^{(\ell-1)p}},
\end{equation*}
where we applied the induction hypothesis on the sequence $\ddiv u\in C_0(\N_0)$, which satisfies $\ddiv u_{n}=0$ if $n<\ell-1$. Using once more the identity $\ddiv=\grad\circ\shift$ and applying the already proven inequality \eqref{eq:p-Hardy_weight} with $\alpha=-(\ell-1)p<0$ and the sequence $\shift u\in C_0(\N_0)$, which satisfies $\shift u_0=u_1=0$, we further estimate
\begin{equation*}
  \lsum{\ell-1} \frac{\abs{\ddiv u_n}^p}{n^{(\ell-1)p}}
    = \lsum{1} \frac{\abs{\grad\shift u_n}^p}{n^{(\ell-1)p}}
    \ge C_{p}(-(\ell-1)p)\lsum{1} \frac{\abs{\shift u_n}^p}{(n+1)^{\ell p}}.
\end{equation*}
Shifting the index in the last sum and noticing that $B_{p}^{(\ell-1)}\,C_{p}(-(\ell-1)p)=B_{p}^{(\ell)}$, see formulas~\eqref{eq:p-birman_constant} and~\eqref{eq:weight_constant}, the above inequalities combine to
\begin{equation*}
  \lsum{\ell} \abs{\grad^\ell u_n}^p
    \ge B_{p}^{(\ell)} \lsum{\ell} \frac{\abs{u_n}^p}{n^{\ell p}}.
\end{equation*}
The proof of inequality~\eqref{eq:p-birman_dis} is complete.
\end{proof}

\section{Continuous \texorpdfstring{$p$}{p}-Birman inequality}
\label{sec:con_birman}

In this section, we derive the continuous $p$-Birman inequality from its discrete counterpart \eqref{eq:p-birman_dis}, following an approach utilized for $p=2$ in \cite{huang-ye_24-hardy}. Beyond the fact that the transition from discrete to continuous inequalities is interesting in itself, this step is motivated by two other objectives. First, it provides an alternative proof of the $p$-Birman inequality to the standard procedure, which naturally avoids the discrete formulation and is indicated in Remark~\ref{rem:fritz}. Second, we take advantage of this passage to prove the optimality of constants in Theorems~\ref{thm:p-birman_dis} and~\ref{thm:weighted_hardy} in the subsequent section.

\begin{thm}[continuous $p$-Birman inequality]
  \label{thm:p-Birman_con}
  Let $\ell\in\N$ and $p>1$. Then for any $\varphi\in C^\infty_0(\R_+)$, we have the inequality
  \begin{equation}
    \label{eq:thm:p-birman_con}
    \int_0^\infty \big|\varphi^{(\ell)}(x)\big|^{p}\dd x
      \ge B_{p}^{(\ell)} \int_0^\infty \frac{\abs{\varphi(x)}^p}{x^{\ell p}}\dd x,
  \end{equation}
  where the constant $B_{p}^{(\ell)}$ is given by formula~\eqref{eq:p-birman_constant} and is sharp.
\end{thm}

 The proof of the optimality of constant $B_{p}^{(\ell)}$ in~\eqref{eq:thm:p-birman_con} is postponed to the next section. Before proving the inequality~\eqref{eq:thm:p-birman_con}, we prepare a lemma which relates the discrete derivatives to the continuous ones.

\begin{lem}\label{lem:dif_der}
 For $N>0$, $n\in\N_0$, and $\varphi\in C_{0}^{\infty}(\R_{+})$, we denote $v_{n}^{(N)}\coloneq\varphi(n/N)$. Then for any $\ell\in\N$, we have
    \begin{equation*}
      \grad^\ell v^{(N)}_n
        = \frac{1}{N^\ell}\,\varphi^{(\ell)}\paren{\frac{n}{N}}+\bigO{\frac{1}{N^{\ell+1}}},\quad\text{as } N\to\infty,
    \end{equation*}
    uniformly for all $n\geq\ell$, i.e. the constant hidden in the Landau $\mc{O}$-symbol is $n$-independent. 
\end{lem}

\begin{proof}
By definition~\eqref{eq:def_div_grad}, we have  
\[
    \grad^\ell v^{(N)}_n 
        =\sum_{k=0}^{\ell}(-1)^{k}\binom{\ell}{k}v_{n-k}^{(N)}.
\]
Moreover, by applying the Taylor–Lagrange formula, we obtain
    \begin{equation*}
      v^{(N)}_{n-k}
        = \varphi\paren{\frac{n-k}{N}}
        = \sum_{j=0}^\ell\frac{1}{j!}\,\varphi^{(j)}\paren{\frac{n}{N}}\paren{\frac{-k}{N}}^j+\bigO{\frac{1}{N^{\ell+1}}},
    \end{equation*}
where the constant in the Landau $\mc{O}$-symbol can be taken $n$-independent since $\varphi\in C_{0}^{\infty}(\R_{+})$. It follows that 
   \begin{align*}
      \grad^\ell v^{(N)}_n
        &= \sum_{k=0}^{\ell}(-1)^{k}\binom{\ell}{k}
          \left[\sum_{j=0}^\ell\frac{1}{j!}\,\varphi^{(j)}\paren{\frac{n}{N}}\paren{\frac{-k}{N}}^j+\bigO{\frac{1}{N^{\ell+1}}}\right] \\
        &= \sum_{j=0}^{\ell}\frac{(-1)^j}{j!}
          \left[\sum_{k=0}^{\ell}(-1)^{k}\binom{\ell}{k}k^j\right]
          \frac{1}{N^j}\varphi^{(j)}\paren{\frac{n}{N}}+\bigO{\frac{1}{N^{\ell+1}}},
    \end{align*}
as $N\to\infty$, uniformly in $n\geq\ell$. Now, it suffices to apply the identity \cite[Eq.~(26.8.6)]{DLMF} 
\[
\sum_{k=0}^{\ell}(-1)^{k}\binom{\ell}{k}k^j=
\begin{cases} 
0 \quad &\mbox{ for } j<\ell, \\
(-1)^{\ell}\ell! \quad &\mbox{ for } j=\ell,
\end{cases}
\]
to the term in the square brackets to finish the proof.
 \end{proof}

\begin{proof}[Proof of inequality~\eqref{eq:thm:p-birman_con}]
First notice that, by scaling of the variable $x$ in both sides of~\eqref{eq:thm:p-birman_con}, it is sufficient to prove~\eqref{eq:thm:p-birman_con} for test functions $\varphi$ supported in $(0,1)$. 

Fix $\varphi\in C^\infty_0(0,1)$, $\ell\in\N$, and define $v^{(N)}_n\coloneq\varphi(n/N)$ for all $n\in\N_0$ and $N>0$ sufficiently large, such that $v^{(N)}_n=0$ for $n<\ell$. Since also $v^{(N)}_n=0$ for all $n>N$, applying Lemma~\ref{lem:dif_der}, we infer that
\begin{align*}
    \lsum{\ell}\abs{\grad^\ell v_n^{(N)}}^p = \sum_{n=1}^N\abs{\grad^\ell v_n^{(N)}}^p
        &= \sum_{n=1}^N\abs{\frac{1}{N^\ell}\,\varphi^{(\ell)}\paren{\frac{n}{N}}+\bigO{\frac{1}{N^{\ell+1}}}}^p \\
        &= \frac{1}{N^{\ell p}}\sum_{n=1}^N \abs{\varphi^{(\ell)}\paren{\frac{n}{N}}}^p+\bigO{\frac{1}{N^{\ell p}}},
  \end{align*}
as $N\to\infty$. For the sequence $u^{(N)}\coloneq N^{\ell-1/p}v^{(N)}$ it follows that
\[
\lsum{\ell}\abs{\grad^\ell u_n^{(N)}}^p=\frac{1}{N}\sum_{n=1}^N \abs{\varphi^{(\ell)}\paren{\frac{n}{N}}}^p+\bigO{\frac{1}{N}}, \quad\text{as } N\to\infty.
\]
Since also $u_{n}^{(N)}=0$ for $n<\ell$ and $n>N$, we may apply the discrete $p$-Birman inequality~\eqref{eq:p-birman_dis}  to obtain
\[
    \sum_{n=1}^N \frac{1}{N}\abs{\varphi^{(\ell)}\paren{\frac{n}{N}}}^p+\bigO{\frac{1}{N}}
        \ge B_{p}^{(\ell)}\,\lsum{\ell}\frac{\abs{u^{(N)}_n}^p}{n^{\ell p}}=B_{p}^{(\ell)}\sum_{n=1}^N \frac{1}{N}\frac{\abs{\varphi\paren{n/N}}^p}{\paren{n/N}^{\ell p}}
\]
for $N\to\infty$. Recognizing the Riemann sums and taking the limit $N\to\infty$ in the above inequality, we arrive at the desired integral inequality \eqref{eq:thm:p-birman_con}.
\end{proof}

\begin{rem}
We remark, without going into details, that one can proceed analogously as in the proof of inequality~\eqref{eq:thm:p-birman_con}, using, this time, the sampling sequence 
\[
u_n^{(N)}\coloneq N^{1-(\alpha+1)/p}\,\varphi\paren{\frac{n}{N}},
\]
to deduce from \eqref{eq:copson} and \eqref{eq:p-Hardy_weight} their continuous counterpart
\begin{equation}\label{eq:rem:weighted_hardy_con}
        \int_0^\infty x^\alpha\abs{\varphi'(x)}^p\dd x
            \ge C_{p}(\alpha)\int_0^\infty x^{\alpha-p}\abs{\varphi(x)}^p\dd x,
\end{equation}
which holds for all $\varphi\in C_0^\infty(\R_{+})$ and $\alpha<p-1$ (the continuous limit takes the same form regardless the sign of $\alpha$).
\end{rem}

\begin{rem}\label{rem:fritz}
 An integral form of the inequality \eqref{eq:rem:weighted_hardy_con} appeared already in the classical book by Hardy, Littlewood, and Pólya \cite[Theorem 330]{har-lit-pol_52}, see also \cite[Theorem~2]{kufner-maligranda-persson_07-hardy_ineq}. The continuous $p$-Birman inequality \eqref{eq:thm:p-birman_con} can be quickly proved through an iterative application of \eqref{eq:rem:weighted_hardy_con}. Specifically, starting with the $\ell$-th derivative of a test function $\varphi$ and selecting $\alpha = -(j-1)p$ in the $j$-th step for $j\in\{1,\dots,\ell\}$, the result follows directly.
\end{rem}

\section{Optimality} \label{sec:optimality}

\subsection{Proofs of optimality of constants}

We establish the optimality of the constants $B_{p}^{(\ell)}$ and $C_{p}(\alpha)$ in Theorems~\ref{thm:p-birman_dis} and~\ref{thm:weighted_hardy}. Due to the transition to the continuous case developed in Section~\ref{sec:con_birman}, it turns out to be sufficient to prove the sharpness of the constants in inequalities \eqref{eq:thm:p-birman_con} and  \eqref{eq:rem:weighted_hardy_con}, which will be done by constructing approximating sequences of test functions for which these inequalities hold as equalities in the limit. Although such a construction is also possible directly for the discrete inequalities \eqref{eq:p-birman_dis} and \eqref{eq:copson}, \eqref{eq:p-Hardy_weight}, the continuous approach is technically simpler: derivatives of monomial functions result in expressions simpler than their discrete counterparts. This elementary but useful idea has been emphasized in \cite{huang-ye_24-hardy}. 

\begin{proof}[Proof of optimality of the $p$-Birman constant $B_{p}^{(\ell)}$ in Theorems \ref{thm:p-birman_dis} and~\ref{thm:p-Birman_con}]
The proof in Section~\ref{sec:con_birman} shows that inequality \eqref{eq:thm:p-birman_con} follows from \eqref{eq:p-birman_dis} regardless of the explicit form of the constant $B_{p}^{(\ell)}$. Consequently, the optimality of the constant $B_{p}^{(\ell)}$ defined by~\eqref{eq:p-birman_constant} in the continuous $p$-Birman inequality \eqref{eq:thm:p-birman_con} immediately implies its optimality in the discrete counterpart \eqref{eq:p-birman_dis}.

For completeness, we verify the optimality of the constant $B_{p}^{(\ell)}$ in the continuous $p$-Birman inequality~\eqref{eq:thm:p-birman_con}, which amounts to proving
\begin{equation}
 \inf_{0\neq\varphi\in\C_{0}^{\infty}\!(\R_{+}\!)}\,\frac{\int_0^\infty |\varphi^{(\ell)}(x)|^{p}\dd x}{\int_0^\infty |\varphi(x)|^{p}\big/x^{\ell p}\dd x} = B_{p}^{(\ell)}
\label{eq:inf_optim}
\end{equation}
by construction of an approximating sequence of test functions $\varphi_N\in C_{0}^{\infty}(\R_{+})$, for which the ratio in the infimum converges to $B_{p}^{(\ell)}$ as $N\to\infty$. Concretely for $N>2$ and $x>0$, we set 
\begin{equation}\label{eq:def_varphi_N}
    \varphi_N(x) \coloneq x^{\ell-1/p}\,\xi_N(x),
\end{equation}
  with the cut-off function $\xi_N\in C_0^\infty(\R_+)$ satisfying
  \begin{equation} \label{eq:xi_propert}
    \xi_N(x)
    =
    \begin{cases}
      1 & \text{if } x\in[2/N,1], \\
      0 & \text{if } x\in\R_{+}\setminus[1/N,2],
    \end{cases}
    \quad\text{and }\quad
    \abs{\xi_N^{(j)}(x)}\le
    \begin{cases}
      CN^j & \text{if } x\in[1/N,2/N], \\
      C & \text{if } x\in[1,2],
    \end{cases}
  \end{equation}
  for all $j\in\{0,\dots,\ell\}$, where $C>0$ is a constant. Such a function can be constructed, for instance, as
  \begin{equation*}
    \xi_N(x) \coloneq \eta(Nx-1)\eta(2-x), 
  \end{equation*}
  where $\eta$ is a smooth function on $\R$ satisfying $\eta(x)=0$ for $x\leq0$ and $\eta(x)=1$ for $x\geq1$. This regularization has been utilized in a slightly different situation in \cite[Note 2.6]{owen_99-hardy-rellich}.

  Using the Leibnitz rule, we estimate the $L^{p}$-norm of the $\ell$-th derivative of $\varphi_N$ as
  \[
  \left(\int_{0}^{\infty}\left|\varphi_{N}^{(\ell)}(x)\right|^{p}\dd x\right)^{1/p} \le \sum_{j=0}^\ell\binom{\ell}{j}\left(1-\frac{1}{p}+j\right)_{\!\ell-j}
  \left(\int_{0}^{\infty}\left|x^{j-1/p}\xi_{N}^{(j)}(x)\right|^{p}\dd x\right)^{1/p}.
  \]
  Employing the properties~\eqref{eq:xi_propert}, we further estimate the integrals in the last sum separately for $j=0$ and $1j\in\{1,\dots,\ell\}$. For $j=0$, we have
  \[
  \int_{0}^{\infty}\left|x^{-1/p}\xi_{N}(x)\right|^{p}\dd x\leq C^{p}\int_{1/N}^{2/N}\frac{\dd x}{x} + \int_{2/N}^{1}\frac{\dd x}{x} + C^{p}\int_{1}^{2}\frac{\dd x}{x} = 2C^{p}\log 2 + \log(N/2),
  \]
  while for $j\in\{1,\dots,\ell\}$, we obtain 
  \[
   \int_{0}^{\infty}\left|x^{j-1/p}\xi_{N}^{(j)}(x)\right|^{p}\dd x\leq C^{p}N^{jp}\int_{1/N}^{2/N}x^{jp-1}\dd x + C^{p}\int_{1}^{2}x^{jp-1}\dd x =2C^{p}\,\frac{2^{jp}-1}{jp}.
  \]
  Recalling formula~\eqref{eq:p-birman_constant}, we obtain the upper bound
  \[
   \left(\int_{0}^{\infty}\left|\varphi_{N}^{(\ell)}(x)\right|^{p}\dd x\right)^{1/p}\leq B_{p}^{(\ell)}\left(A_0 + \log(N/2)\right)^{1/p}+A_1,
  \]
  where $A_0$ and $A_1$ are $N$-independent constants. Conversely, we find
  \[
  \int_{0}^{\infty}\frac{|\varphi_{N}(x)|^{p}}{x^{\ell p}}\dd x\geq \int_{2/N}^{1}\frac{\dd x}{x}=\log(N/2).
  \]
  Altogether, bearing~\eqref{eq:thm:p-birman_con} in mind, we have
  \[
   B_{p}^{(\ell)}\leq\frac{\int_0^\infty |\varphi_N^{(\ell)}(x)|^{p}\dd x}{\int_0^\infty |\varphi_N(x)|^{p}\big/x^{\ell p}\dd x}\leq
   \frac{\left(B_{p}^{(\ell)}\left(A_0+\ln(N/2)\right)^{1/p}+A_1\right)^{\! p}}{\ln(N/2)} \longrightarrow B_{p}^{(\ell)},\;\mbox{ as } N\to\infty,
  \]
  which yields~\eqref{eq:inf_optim}. The proof of Theorem~\ref{thm:p-birman_dis}, and also Theorem~\ref{thm:p-Birman_con}, is complete.
\end{proof}

\noindent
\emph{Proof of optimality of the Copson constant $C_{p}(\alpha)$ in Theorem~\ref{thm:weighted_hardy} and inequality~\eqref{eq:rem:weighted_hardy_con}}.
We omit the details here, as the proof of the optimality of the constant
$C_{p}(\alpha)$ in \eqref{eq:rem:weighted_hardy_con}, and consequently in
Theorem~\ref{thm:weighted_hardy} (as well as in~\eqref{eq:copson}), follows a similar
argument to the one detailed above. In this case, the test functions
\eqref{eq:def_varphi_N} are replaced by
\[
\varphi_{N}(x)\coloneq x^{1-(\alpha+1)/p}\,\xi_N(x).
\]
The sharpness within a space of non-smooth functions is also proved in \cite[Theorem 2(iii)]{kufner-maligranda-persson_07-hardy_ineq}.
\qed

\subsection{Remarks on optimality of weights}

Recent developments in the study of Hardy-type inequalities have shifted focus toward stronger notions of optimality that account for the entire weight function rather than merely the sharpness of the multiplicative constant. These notions are centered on the concept of \emph{criticality}. A weight sequence or function appearing on the right-hand side of inequalities such as \eqref{eq:p-Hardy_dis} or \eqref{eq:p-Hardy_cont} is said to be \emph{critical} if it admits no pointwise improvement on $\mathbb{N}$ (in the discrete case) or almost everywhere on $\mathbb{R}_{+}$ (in the continuous case). This subsection briefly summarizes the state of the art regarding critical weights in this context.

Despite the optimality of the $p$-Hardy constant $B_p^{(1)}$ in \eqref{eq:p-Hardy_dis} and \eqref{eq:p-Hardy_cont}, the $p$-Hardy weight function $B_p^{(1)}/x^{2p}$ in \eqref{eq:p-Hardy_cont} is critical, whereas its discrete counterpart $B_p^{(1)}/n^{2p}$ in \eqref{eq:p-Hardy_dis} is \textbf{not}. This distinction was first observed for $p=2$ in \cite{keller-pinchover-pogorzelski_18_hardy_on_graphs}, where a critical discrete Hardy weight was identified; see also \cite{kel-pin-pog_18, kre-sta_22}. Critical discrete $p$-Hardy weights for general $p>1$ were found in \cite{fisher-keller-pogorzelski_23-p-hardy} and further discussed in \cite{stampach-waclawek_25-herglotz}. Additionally, critical improvements of the Copson inequality \eqref{eq:copson} for specific powers are provided in \cite{das-man-pau_25}. Further investigations into the optimality of discrete Hardy inequalities across various settings include \cite{das-fue_26, fis_23, fis_24, kel-nit_23, kel-pin-pog_20}.

Close inspection of the estimates in the proof of the discrete $p$-Birman inequality in Section~\ref{sec:proofs} reveals that the discrete $p$-Birman weight in \eqref{eq:p-birman_dis} also admits an improvement. For $p=2$, the first improvements of the discrete Rellich weight ($\ell=2$) were reported in \cite{ger-kre-sta_25}. Critical discrete Birman weights for arbitrary $\ell \in \mathbb{N}$ and $p=2$ were obtained recently in \cite{stampach-waclawek_24-birman}. Finally, critical discrete $p$-Birman weights for the general case $p>1$, alongside a theoretical framework generalizing~\cite{stampach-waclawek_24-birman}, will be presented in the forthcoming work \cite{stampach-waclawek_26prep}.
 
\subsection*{Acknowledgments}
We are grateful to Fritz Gesztesy for his helpful comments on the origins of the continuous $p$-Birman inequality, which is mentioned in Remark~\ref{rem:fritz}.

\bibliographystyle{acm}

\begin{thebibliography}{10}

\bibitem{bennett_87}
{\sc Bennett, G.}
\newblock {Some elementary inequalities}.
\newblock {\em Quart. J. Math. Oxford Ser. (2) 38}, 152 (1987), 401--425.

\bibitem{birman_61}
{\sc Birman, M.~{\v S}.}
\newblock {On the spectrum of singular boundary-value problems}.
\newblock {\em Mat. Sb. (N.S.) 55 (97)} (1961), 125--174.

\bibitem{copson_28}
{\sc Copson, E.~T.}
\newblock {Note on {S}eries of {P}ositive {T}erms}.
\newblock {\em J. London Math. Soc. 3}, 1 (1928), 49--51.

\bibitem{das-manna_23-weight_hardy}
{\sc Das, B., and Manna, A.}
\newblock {On the improvements of Hardy and Copson inequalities}.
\newblock {\em Rev. Real Acad. Cienc. Exactas Fis. Nat. Ser. A-Mat. 117}, 2
  (2023).

\bibitem{DasManna2025_ImprovedCopson}
{\sc Das, B., and Manna, A.}
\newblock {An improved Copson inequality}.
\newblock {arXiv:2508.00388 [math.CA]} (2025).

\bibitem{das-man-pau_25}
{\sc Das, B., Manna, A., and Paul, T.}
\newblock {Improvement of the discrete weighted variant {H}ardy's inequality
  and criticality of the improved weight}.
\newblock {\em J. Geom. Anal. 35}, 12 (2025), 22 pp.

\bibitem{das-fue_26}
{\sc Das, U., and de~la Fuente-Fern\'andez, R.}
\newblock {An optimal fractional {H}ardy inequality on the discrete half-line}.
\newblock {\em Calc. Var. Partial Differential Equations 65}, 2 (2026).

\bibitem{DLMF}
{\it NIST Digital Library of Mathematical Functions}.
\newblock http://dlmf.nist.gov/, Release 1.1.5 of 2022-03-15.
\newblock F.~W.~J. Olver, A.~B. {Olde Daalhuis}, D.~W. Lozier, B.~I. Schneider,
  R.~F. Boisvert, C.~W. Clark, B.~R. Miller, B.~V. Saunders, H.~S. Cohl, and
  M.~A. McClain, eds.

\bibitem{fis_23}
{\sc Fischer, F.}
\newblock {A non-local quasi-linear ground state representation and criticality
  theory}.
\newblock {\em Calc. Var. Partial Differential Equations 62}, 5 (2023), 33 pp.

\bibitem{fis_24}
{\sc Fischer, F.}
\newblock {On the optimality and decay of {$p$}-{H}ardy weights on graphs}.
\newblock {\em Calc. Var. Partial Differential Equations 63}, 7 (2024), 37 pp.

\bibitem{fisher-keller-pogorzelski_23-p-hardy}
{\sc Fischer, F., Keller, M., and Pogorzelski, F.}
\newblock {An Improved Discrete p-Hardy Inequality}.
\newblock {\em Integral Equations and Operator Theory 95}, 24 (2023).

\bibitem{frank-seiringer_08-ineq}
{\sc Frank, R.~L., and Seiringer, R.}
\newblock {Non-linear ground state representations and sharp {H}ardy
  inequalities}.
\newblock {\em Journal of Functional Analysis 255}, 12 (2008), 3407--3430.

\bibitem{ger-kre-sta_25}
{\sc Gerhat, B., Krej\v{c}i\v{r}\'{i}k, D., and \v{S}tampach, F.}
\newblock {An improved discrete {R}ellich inequality on the half-line}.
\newblock {\em Israel J. Math. 268}, 1 (2025), 45--72.

\bibitem{ges-lit-mic-wel_18}
{\sc Gesztesy, F., Littlejohn, L.~L., Michael, I., and Wellman, R.}
\newblock {On {B}irman's sequence of {H}ardy-{R}ellich-type inequalities}.
\newblock {\em J. Differential Equations 264}, 4 (2018), 2761--2801.

\bibitem{glazman}
{\sc Glazman, I.~M.}
\newblock {\em {Direct methods of qualitative spectral analysis of singular
  differential operators}}.
\newblock Daniel Davey \& Co., Inc., New York, 1966.

\bibitem{har-lit-pol_52}
{\sc Hardy, G.~H., Littlewood, J.~E., and P\'olya, G.}
\newblock {\em {Inequalities}}.
\newblock Cambridge University Press, 1952.

\bibitem{huang-ye_24-hardy}
{\sc Huang, X., and Ye, D.}
\newblock {One-dimensional sharp discrete {H}ardy-{R}ellich inequalities}.
\newblock {\em J. Lond. Math. Soc. 109}, 1 (2024), 26 pp.

\bibitem{kel-nit_23}
{\sc Keller, M., and Nietschmann, M.}
\newblock {Optimal {H}ardy inequality for fractional {L}aplacians on the
  integers}.
\newblock {\em Ann. Henri Poincar\'e 24}, 8 (2023), 2729--2741.

\bibitem{kel-pin-pog_18}
{\sc Keller, M., Pinchover, Y., and Pogorzelski, F.}
\newblock {An improved discrete {H}ardy inequality}.
\newblock {\em Amer. Math. Monthly 125}, 4 (2018), 347--350.

\bibitem{keller-pinchover-pogorzelski_18_hardy_on_graphs}
{\sc Keller, M., Pinchover, Y., and Pogorzelski, F.}
\newblock {Optimal {H}ardy inequalities for {S}chr\"{o}dinger operators on
  graphs}.
\newblock {\em Comm. Math. Phys. 358}, 2 (2018), 767--790.

\bibitem{kel-pin-pog_20}
{\sc Keller, M., Pinchover, Y., and Pogorzelski, F.}
\newblock {Criticality theory for {S}chr\"odinger operators on graphs}.
\newblock {\em J. Spectr. Theory 10}, 1 (2020), 73--114.

\bibitem{kre-sta_22}
{\sc Krej\v{c}i\v{r}\'{i}k, D., and \v{S}tampach, F.}
\newblock {A sharp form of the discrete {H}ardy inequality and the
  {K}eller-{P}inchover-{P}ogorzelski inequality}.
\newblock {\em Amer. Math. Monthly 129}, 3 (2022), 281--283.

\bibitem{kufner-malingrada-oersson_06-prehistory}
{\sc Kufner, A., Maligranda, L., and Persson, L.-E.}
\newblock {The prehistory of the {H}ardy inequality}.
\newblock {\em Amer. Math. Monthly 113}, 8 (2006), 715--732.

\bibitem{kufner-maligranda-persson_07-hardy_ineq}
{\sc Kufner, A., Maligranda, L., and Persson, L.-E.}
\newblock {\em {The Hardy Inequality: About Its History and Some Related
  Results}}.
\newblock Vydavatelský Servis, Pilsen, Czech Republic, 2007.

\bibitem{owen_99-hardy-rellich}
{\sc Owen, M.~P.}
\newblock {The {H}ardy--{R}ellich inequality for polyharmonic operators}.
\newblock {\em Proceedings of the Royal Society of Edinburgh, Section A:
  Mathematics 129} (1999), 825--839.

\bibitem{rellich_56}
{\sc Rellich, F.}
\newblock {Halbbeschr\"{a}nkte {D}ifferentialoperatoren h\"{o}herer {O}rdnung}.
\newblock In {\em Proceedings of the {I}nternational {C}ongress of
  {M}athematicians, 1954, {A}msterdam, vol. {III}} (1956), North-Holland Publishing Co., Amsterdam, pp.~243--250.

\bibitem{stampach-waclawek_26prep}
{\sc {{\v{S}}tampach}, F., and {Waclawek}, J.}
\newblock {Optimal discrete $p$-Hardy-Rellich-Birman inequalities}.
\newblock {\em In preparation} (2026).

\bibitem{stampach-waclawek_25-herglotz}
{\sc \v{S}tampach, F.~s., and Waclawek, J.}
\newblock {A {H}erglotz--{N}evanlinna function from the optimal discrete
  {$p$}-{H}ardy weight}.
\newblock {\em Proc. Amer. Math. Soc. 154}, 2 (2026), 807--819.

\bibitem{stampach-waclawek_24-birman}
{\sc Štampach, F., and Waclawek, J.}
\newblock Optimal discrete {H}ardy--{R}ellich--{B}irman inequalities.
\newblock {\em Journal d'Analyse Mathématique} (2025).
\newblock Accepted for publication.

\end{thebibliography}

\end{document}